\documentclass[a4paper, reqno, 11pt]{amsart}

\usepackage[english]{babel}
\usepackage{amsmath}
\usepackage{mathrsfs}
\usepackage{amssymb}
\usepackage{amsthm}
\usepackage{enumerate}
\usepackage{ifthen}
\usepackage{bbm}
\usepackage{color}
\provideboolean{shownotes} 
\setboolean{shownotes}{true}
\usepackage{hyperref}
\usepackage{graphicx}
\usepackage{pgfplots}
\usepackage{tikz,tikz-3dplot} 
\usepackage{mathtools}
\usepackage{caption}[2004/07/16]

\pgfplotsset{compat=1.8}

\newcommand{\margnote}[1]{
\ifthenelse{\boolean{shownotes}}%
{\marginpar{\raggedright\tiny\texttt{#1}}}%
{}%
}

\newcommand{\hole}[1]{
\ifthenelse{\boolean{shownotes}}%
{\begin{center} \fbox{ \rule {.25cm}{0cm}
\rule[-.1cm]{0cm}{.4cm} \parbox{.85\textwidth}{\begin{center}
\texttt{#1}\end{center}} \rule {.25cm}{0cm}}\end{center}}
{}
}
\newtheorem{thm}{Theorem}[section]
\newtheorem{prop}[thm]{Proposition}

\theoremstyle{definition}
\newtheorem{defn}[thm]{Definition}

\newcommand{\e}{\varepsilon}		       
\newcommand{\R}{\mathbb{R}}

\newcommand{\Z}{\mathbb{Z}}

\newcommand{\sgn}{\text{sgn}}

\newcommand{\de}{\mathrm{d}}

\numberwithin{equation}{section}

\begin{document}

\title[Smooth approximations and selection of flows]{On smooth approximations of rough vector fields and the selection of flows}

\author[G. Ciampa]{Gennaro Ciampa}
\address[G. Ciampa]{GSSI - Gran Sasso Science Institute\\ Viale Francesco Crispi 7 \\67100 L'Aquila \\Italy \& Department Mathematik Und Informatik\\ Universit\"at Basel \\Spiegelgasse 1 \\CH-4051 Basel \\ Switzerland}
\email[]{\href{gennaro.ciampa@}{gennaro.ciampa@gssi.it},\href{gennaro.ciampa@2}{gennaro.ciampa@unibas.ch}}

\author[G. Crippa]{Gianluca Crippa}
\address[G. Crippa]{Department Mathematik Und Informatik\\ Universit\"at Basel \\Spiegelgasse 1 \\CH-4051 Basel \\ Switzerland}
\email[]{\href{gianluca.crippa@}{gianluca.crippa@unibas.ch}}

\author[S. Spirito]{Stefano Spirito}
\address[S. Spirito]{DISIM - Dipartimento di Ingegneria e Scienze dell'Informazione e Matematica\\ Universit\'a degli Studi dell'Aquila \\Via Vetoio \\ 67100 L'Aquila \\ Italy}
\email[]{\href{stefano.spirito@}{stefano.spirito@univaq.it}}

\begin{abstract}
In this work we deal with the selection problem of flows of an irregular vector field. We first summarize an example from \cite{CCS} of a vector field $b$ and a smooth approximation $b_\e$ for which the sequence $X^\e$ of flows of $b_\e$ has subsequences converging to different flows of the limit vector field $b$. Furthermore, we give some heuristic ideas on the selection of a subclass of flows in our specific case.
\end{abstract}

\maketitle

\section{Introduction and notations}
Consider the system of ordinary differential equations
\begin{equation}
\begin{dcases}
\frac{\de}{\de t}X(t,x)=b(t,X(t,x)), \\
X(0,x)=x,
\end{dcases}
\label{eq:ode}
\end{equation}
where $(t,x)\in (0,T)\times\R^{d}$ are the independent variables, with $T<\infty$, $b:(0,T)\times\R^{d}\to\R^{d}$ is a given vector field and $X:(0,T)\times\R^{d}\to\R^{d}$ is the unknown. A solution $X$ of \eqref{eq:ode} is called \emph{flow} of $b$. The well posedness of \eqref{eq:ode} is a well known result when the vector field $b$ is globally Lipschitz in space uniformly in time. The system \eqref{eq:ode} is strictly connected to the Cauchy problem for the linear transport equation
\begin{equation}
\begin{cases}
\partial_t u+ b \cdot \nabla u=0, \\
u|_{t=0}=u_0,
\end{cases}
\label{eq:te}
\end{equation}
since in a smooth setting, the unique solution of \eqref{eq:te} is given by the formula $u(t,x)=u_{0}((X(t,\cdot))^{-1}(x))$, where $X$ is the unique flow of $b$. \\
Besides the theoretical interest, due to applications to several equations from mathematical physics the setting of smooth vector fields is too restrictive and a theory under assumptions of lower regularity was developed in the last years.
Exploiting the connection between \eqref{eq:te} and \eqref{eq:ode}, DiPerna and Lions in \cite{DPL} proved the well posedness of \eqref{eq:te} under hypotesis of Sobolev regularity for the vector field and bounded divergence. As a consequence of their result, they proved well posedness of \eqref{eq:ode} under the same hypotesis. Similarly, Ambrosio in \cite{Am} improved the result of \cite{DPL} to the case of BV regularity and bounded divergence for $b$. On the other hand, a well posedness theory based only on \emph{a prori} estimates of the flow was developed in \cite{CDL} for $W^{1,p}$ vector field with $p>1$ and in \cite{BC},\cite{CNSS} for the case $p=1$ and vector fields whith gradient given by a singular integral of a $L^1$ function. This latter is a class of interest in the context of 2D Euler equations. More recently Nguyen in \cite{N} improved the result to vector fields which can be represented as singular integral of a function in $BV$. \\
Various counterexamples show that weak differentiability assumptions on the vector field are in general necessary in order to obtain well posedness, see for instance \cite{DeP}, \cite{DPL}. For a general survey on this topic, we refer to \cite{AC}.
The aim of this note is to discuss the selection problem for solutions of \eqref{eq:ode} in a low regularity setting. To better explain what we mean by selection, let us first recall some preliminary notations and definitions. \\
We denote by $\mathscr{L}^d$ the Lebesgue measure on $\R^d$. If $f:\R^d\to\R$ is a Borel map we denote by $f\#\mathscr{L}^d$ the \emph{push forward}, that is, the measure defined by the following relation
$$
f\#\mathscr{L}^d(E)=\mathscr{L}^d(f^{-1}(E)) \hspace{1cm}\mathrm{for \ every \ Borel \ set} \ E\subset\R^d.
$$
The definition of flow of a vector field $b$, when $b$ is not smooth, is the following:
\begin{defn}
Let $b\in L^1((0,T);L^1_{\mathrm{loc}}(\R^d;\R^d))$ be given. We say that $X:(0,T)\times\R^d\rightarrow\R^d$ is a regular Lagrangian flow associated to $b$ if
\begin{enumerate}
\item for a.e. $x\in\R^d$ the map $t\rightarrow X(t,x)$ is an absolutely continuous integral solution of the ordinary differential equation
\begin{equation}
\begin{dcases}
\frac{\de}{\de t}X(t,x)=b(t,X(t,x)), \\
X(0,x)=x,
\end{dcases}
\end{equation}
\item there exists a constant $L$ indipendent of $t$ such that
\begin{equation}\label{eq:incom}
X(t,\cdot)\# \mathscr{L}^d\leq L \mathscr{L}^d.
\end{equation}
\end{enumerate}
\end{defn}
If the vector field is divergence-free, $L$ can be taken to be $1$ and $\eqref{eq:incom}$ is an equality. This means that the flow $X$ is measure preserving.
Condition \eqref{eq:incom} is a first selection: we only consider among all solutions of \eqref{eq:ode} those that do not “compress” too much the Lebesgue measure. This selection is necessary in the theory since it is not known if there is uniqueness in the class of flows that can compress the Lebesgue measure, even under assumptions of weak differentiability and zero divergence for the vector field.\\
The paper is divided as follows. In Section 2 we give a precise statement for the selection problem and we give an example of a vector field and of an approximation for which the selection is not true. In Section 3 we characterize a class of flows through measure preserving maps of the unit circle. Finally in Section 4 we introduce a new question about the selection of a subset of the set of all flows and we give some ideas and heuristics about what we can expect.

\section{The problem of selection}
Let us consider a weakly differentiable vector field $b$ which falls into the class of well-posedness like those discussed in the introduction. To prove the existence of solutions of \eqref{eq:ode}, the natural approach is to rely on a compactness argument for an approximating sequence $X^\e$. This latter is usually constructed as the (unique) flow of a smooth approximation $b_\e$ of $b$, see \cite{AC}. Consider, instead, a vector field $b$ that has more than one regular Lagrangian flow and let $b_\e$ be a smooth approximation of $b$. Consider the solution $X^\e$ of the ODE relative to $b_\e$ and assume that $X^\e$ converges to a regular Lagrangian flow $X$ of $b$. We wonder if for every approximation $b_\e$ the corresponding flows $X^\e$ can converge to only one regular Lagrangian flow: if this were true, this procedure could be considered as a selection principle for the flows of an irregular vector field. We can summarize the previous discussion in the following: question
\begin{itemize}
\item[(Q1)]\emph{Does the approximation procedure obtained by smoothing the vector field select a unique solution of \eqref{eq:ode}?}
\end{itemize}
In \cite{CCS} we give a negative answer to the previous question showing a counterexample. Precisely, we consider this vector field, which is a 3D analogous of an example of DiPerna and Lions \cite{DPL}:
\begin{equation}
b(x,y,z)=
\begin{dcases}
\left(- \sgn(z) \frac{x}{|z|^2},- \sgn(z) \frac{y}{|z|^2}, - \frac{2}{|z|} \right)  & \mathrm{if} \ x\in P,\\ 
(0,0,0) & \mathrm{otherwise},
\end{dcases}
\label{eq:b}
\end{equation}
where $P\subset\R^{3}$ denotes the set 
\[
P=P^+\cup P^-=\{ (x,y,z)\in\mathbb{R}^3: x^2+y^2\leq z \}\cup \{\ (x,y,z)\in\mathbb{R}^3:x^2+y^2\leq -z\},
\]
the union of two symmetric paraboloids. \\ The vector field $b$ is divergence-free and it is out of the class of uniqueness of solutions of \eqref{eq:ode}. In particular, observe that $b$ is not in any Sobolev space $W^{1,p}(\R^3)$ or in $BV(\R^3)$. 
\\
We want to define two different regular Lagrangian flows $\bar{X},\tilde{X}$ of $b$ and, since we are considering flows defined almost everywhere, we need to define $\bar{X},\tilde{X}$ only on $\R^3\setminus \{0\}$. We start for $\mathbf{x}\in\R^3\setminus P$: in this region the vector field is identically $0$ so that we define $\bar{X},\tilde{X}$ simply as
$$
\bar{X}(t,\mathbf{x})=\mathbf{x}=\tilde{X}(t,\mathbf{x}) \hspace{1cm}\forall t\geq 0.
$$
If $\mathbf{x}=(x,y,z)\in P^-$ we define $\bar{X},\tilde{X}$ as
\begin{equation}\label{x}
\begin{dcases}
\bar{X}_1(t,x,z)=\tilde{X}_1(t,x,z)=\frac{x}{\sqrt{-z}}\sqrt[4]{z^2+4t} \\
\bar{X}_2(t,y,z)=\tilde{X}_2(t,y,z)=\frac{y}{\sqrt{-z}}\sqrt[4]{z^2+4t} \\
\bar{X}_3(t,z)=\tilde{X}_3(t,z)=-\sqrt{z^2+4t}
\end{dcases}
\hspace{0.7cm} \forall \ t\geq 0.
\end{equation}
Finally, when $\mathbf{x}=(x,y,z)\in P^+$ define the flows as
\begin{equation}\label{x}
\begin{dcases}
\bar{X}_1(t,x,z)=\tilde{X}_1(t,x,z)=\frac{x}{\sqrt{z}}\sqrt[4]{z^2-4t} \\
\bar{X}_2(t,y,z)=\tilde{X}_2(t,y,z)=\frac{y}{\sqrt{z}}\sqrt[4]{z^2-4t} \\
\bar{X}_3(t,z)=\tilde{X}_3(t,z)=\sqrt{z^2-4t}
\end{dcases}
\hspace{0.7cm} \mathrm{for} \ t \in \left[0,\frac{z^2}{4} \right].
\end{equation}
At time $t=\frac{z^2}{4}$ the trajectories reach the origin and then one possible way to extend them for later times is
\begin{equation}\label{eq:thetasol}
\begin{dcases}
\bar{X}_1(t,x,z)=\frac{x}{\sqrt{z}}\sqrt[4]{4t-z^2} \cos\Theta - \frac{y}{\sqrt{z}}\sqrt[4]{4t-z^2} \sin\Theta\\
\bar{X}_2(t,y,z)=\frac{x}{\sqrt{z}}\sqrt[4]{4t-z^2} \sin\Theta+\frac{y}{\sqrt{z}}\sqrt[4]{4t-z^2} \cos\Theta \\
\bar{X}_3(t,z)=-\sqrt{4t-z^2}
\end{dcases}
\hspace{0.7cm} t \geq\frac{z^2}{4},
\end{equation}
while
\begin{equation}\label{eq:phisol}
\begin{dcases}
\tilde{X}_1(t,x,z)=\frac{x}{\sqrt{z}}\sqrt[4]{4t-z^2} \cos\Phi - \frac{y}{\sqrt{z}}\sqrt[4]{4t-z^2} \sin\Phi\\
\tilde{X}_2(t,y,z)=\frac{x}{\sqrt{z}}\sqrt[4]{4t-z^2} \sin\Phi+\frac{y}{\sqrt{z}}\sqrt[4]{4t-z^2} \cos\Phi \\
\tilde{X}_3(t,z)=-\sqrt{4t-z^2}
\end{dcases}
\hspace{0.7cm} t \geq\frac{z^2}{4},
\end{equation}
where $\Theta,\Phi\in (0,2\pi]$ and $\Theta\neq\Phi$. An easy computation shows that $\bar{X},\tilde{X}$ are two different regular Lagrangian flows of $b$. We call those kind of solutions respectively $X^\Theta,X^\Phi$, where $\Theta$ and $\Phi$ represent a rotation in the $xy$ plane. Heuristically, we can define this kind of flows as a consequence of the fact that the trajectories once they reach the origin can come out arbitrarily.
In \cite{CCS} one of our main results is the following:
\begin{thm}\label{teo:main2}
There exists a sequence of vector fields $b_n\in C^\infty(\R^{3})$ such that:
\begin{itemize}
\item $b_n$ is divergence-free;
\item $b_n\to b$ in $L_{\mathrm{loc}}^1(\R^{3})$;
\item the sequence $X^n$ of regular Lagrangian flows of $b_n$ has two different subsequences converging in $L^\infty((0,T);L^1_{\mathrm{loc}}(\R^3)))$ to two different regular Lagrangian flows of $b$.
\end{itemize}
\end{thm}
\begin{figure}[h!]
\centering
\includegraphics[width=0.42\textwidth]{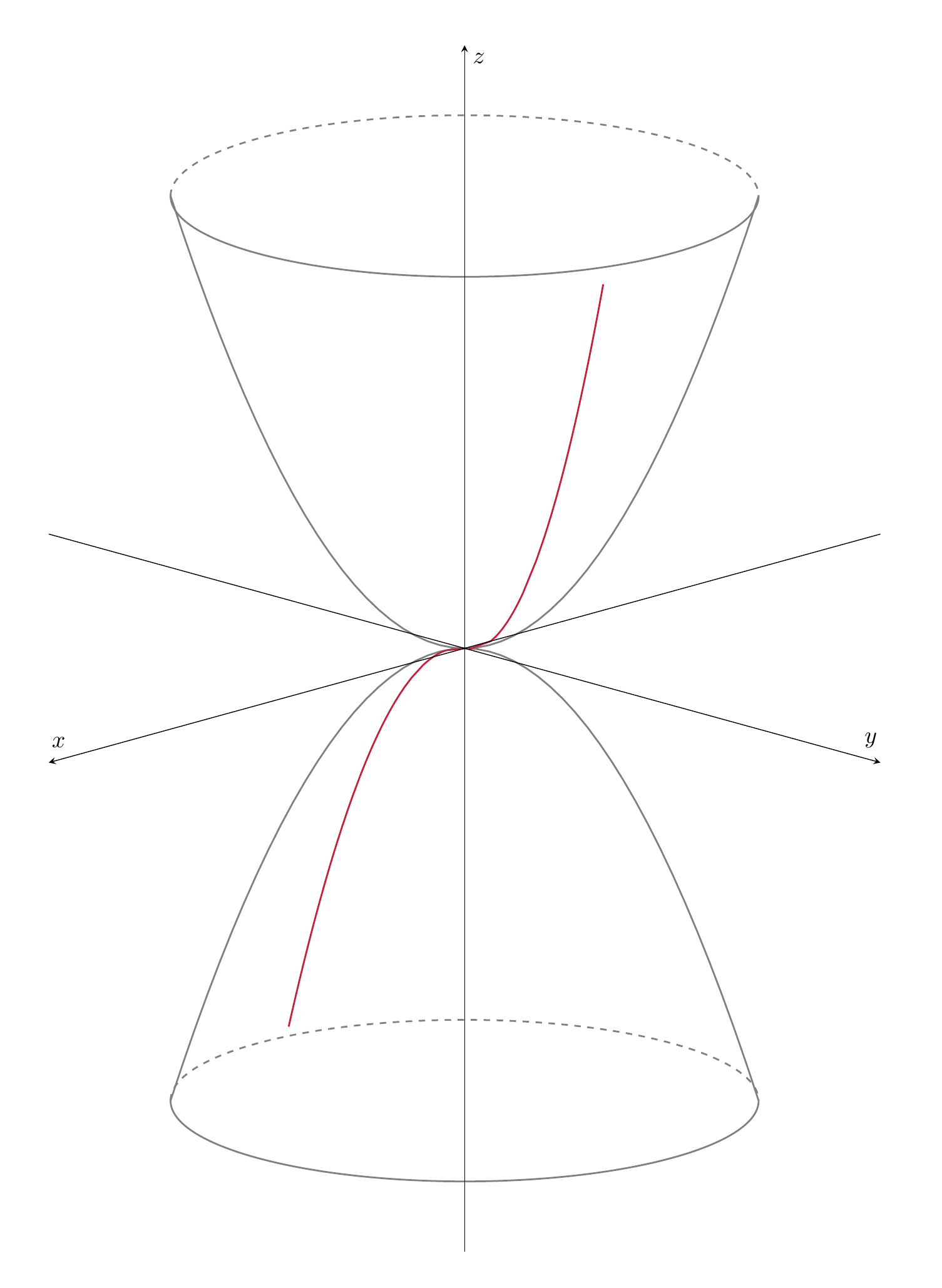}
\caption{An example of solution $X^\Theta$}
\end{figure}
In the proof of Theorem \ref{teo:main2}, given $\Theta,\Phi\in(0,2\pi]$ with $\Theta\neq\Phi$, we construct an explicit approximation which has two different subsequences converging respectively to $X^\Theta$ and $X^\Phi$. The strategy of the approximation is based on smoothing $b$ nearby the origin and forcing the trajectories to rotate very fast along a given helix. We basically modify $b$ in a small region with contains the singularity and leave the rest unchanged. \\
\begin{figure}[h!]\label{fig:clessidra}
\centering
\includegraphics[width=0.6\textwidth]{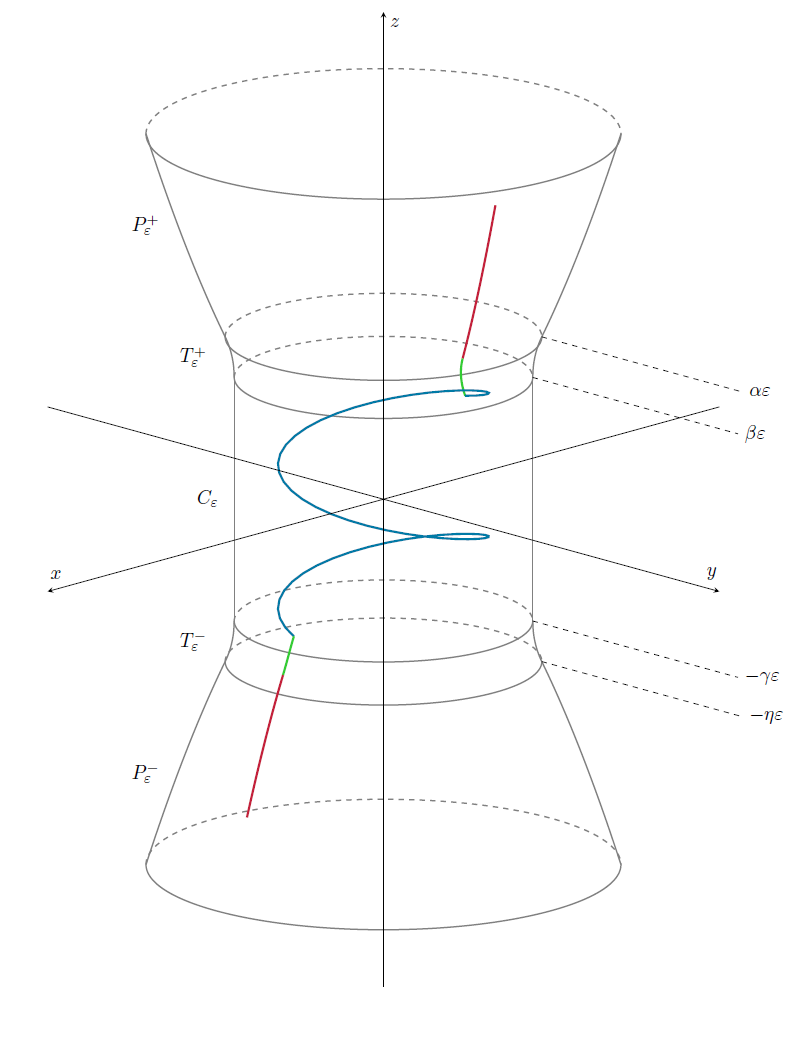}
\caption{The figure represents an approximated trajectory in the construction of the proof of Theorem \ref{teo:main2}}
\end{figure} \\
The theorem answers question (Q1) in the negative.
However with our approach we are able to obtain only solutions of the form $X^\Theta$. Indeed, note that another possible way to define a flow for $\mathbf{x}\in P^+$ is the following:
\begin{equation}\label{eq:xpsi1}
\begin{dcases}
X_1(t,r,\theta,z)=\frac{r}{\sqrt{z}}\sqrt[4]{z^2-4t}\ \cos\theta \\
X_2(t,r,\theta,z)=\frac{r}{\sqrt{z}}\sqrt[4]{z^2-4t}\ \sin\theta \\
X_3(t,z)=\sqrt{z^2-4t}
\end{dcases}
\hspace{0.7cm} \mathrm{for} \ t \in \left[0,\frac{z^2}{4} \right],
\end{equation}
and
\begin{equation}\label{eq:xpsi2}
\begin{cases}
X_1(t,r,\theta,z)=\displaystyle\frac{r}{\sqrt{z}}\sqrt[4]{4t-z^2} \ \cos\psi(\theta)\\
X_2(t,r,\theta,z)=\displaystyle\frac{r}{\sqrt{z}}\sqrt[4]{4t-z^2} \ \sin\psi(\theta) \\ 
X_3(t,z)=-\sqrt{4t-z^2}
\end{cases}
\hspace{0.7cm} \mathrm{for} \  t \geq\frac{z^2}{4},
\end{equation}
where the map $\psi:[0,2\pi]\to[0,2\pi]$ is arbitrary and $(r,\theta,z)$ denote the cylindrical coordinates in $\R^3$. It is easy to check that the map in \eqref{eq:xpsi1},\eqref{eq:xpsi2} is a solution of the ODE relative to $b$; we call $X_\psi$ such a map. It will turn out to be useful the flow on $P\setminus\{0\}$ and not only in $\mathring{P}$ although we deal with functions defined almost everywhere with respect to the 3D Lebesgue measure. The reason for that lies in the fact that for our purpose we will compute $X$ on $\partial P\setminus\{0\}$; this would not make sense without a suitable definition of $X$ on the boundary of $P$. Such definition is made accordingly to the everywhere definition of $b$. In the next section we will discuss the conditions that the map $\psi$ has to satisfied in order for $X_\psi$ to be a regular Lagrangian flow of $b$.\\

\section{Regular Lagrangian flows and measure preserving map on the circle}
In this section we prove that solutions of the form $X_\psi$ are regular Lagrangian flows of $b$ when $\psi$ is a measure preserving map. Before doing this, note that the map $X_\psi$ associated to $\psi(\theta)=\alpha$, where $\alpha \in(0,2\pi]$ is fixed, is a solution of the ODE but it does not preserve the 3D Lebesgue measure and then it is not a regular Lagrangian flow. \\
Now we recall the definition of a mesure preserving map on the unit circle.
\begin{defn}
Let $\psi:\mathbb{S}^1\rightarrow\mathbb{S}^1$ be a measurable map, where $\mathbb{S}^1=\R /2\pi\Z$ is the unit circle with the 1D Lebesgue measure. The map $\psi$ is called measure preserving if $$\psi\#\mathscr{L}^1=\mathscr{L}^1.$$
\end{defn}
We identify $\mathbb{S}^1\sim [0,2\pi]$ and we define the set $\mathscr{M}$ as
$$
\mathscr{M}:=\{ \psi:[0,2\pi]\to[0,2\pi]: \psi \ \mathrm{satisfies \ Definition \ 3.1} \}.
$$
Moreover, define the maps
$$
I_\pm:\theta\in [0,2\pi]\rightarrow (\cos\theta,\sin\theta,\pm 1)\in \R^3.
$$
\begin{prop}
Given a regular Lagrangian flow $X$ there exists $\psi\in\mathscr{M}$ such that $X=X_\psi$. Viceversa given $\psi\in\mathscr{M}$ there exists a unique regular Lagrangian flow $X$ such that $X=X_\psi$.
\end{prop}
\begin{proof}
Consider a regular Lagrangian flow $X(t,\mathbf{x})$ and define
$$
\psi(\theta)=I_-^{-1}\left( X\left(\frac{1}{2},I_+(\theta)\right)\right)
\hspace{1cm}\theta\in[0,2\pi].
$$
We need to show that such a map preserves the 1D Lebesgue measure: consider a Borel set $E\subseteq[0,2\pi]$ and define $\mathbf{E}$ as the set
\[
\mathbf{E}=\{ (\rho,\theta,z):\theta\in E, \rho\in [0,\sqrt{z}], z\in [-1,0] \}.
\]
A straightforward computation shows that
\[
X^{-1}\left(\frac{1}{2},\cdot\right)(\mathbf{E})=\{ (\rho,\theta,z):\theta\in \psi^{-1}(E), \rho\in [0,\sqrt{z}], z\in [1,\sqrt{2}] \}
\]
and
\begin{equation}\label{mpm}
\mathscr{L}^1(\psi^{-1}(E))=4\mathscr{L}^3\left(X^{-1}\left(\frac{1}{2},\cdot\right)(\mathbf{E})\right)=4\mathscr{L}^3(\mathbf{E})=\mathscr{L}^1(E),
\end{equation}
hence $\psi$ is measure preserving. \\
We now prove the other implication. Consider a measure preserving map $\psi$, a point $\mathbf{x}\in \mathbb{S}^1\times\{1\}$ and solve the system
\begin{equation}\label{odepsi}
\begin{cases}
\dot{X}(t,\mathbf{x})=b(X(t,\mathbf{x})) \\
X(0,\mathbf{x})=\mathbf{x}\\
X\left(\frac{1}{2},\mathbf{x}\right)=I_-\left(\psi(I_+^{-1}(\mathbf{x}))\right)
\end{cases}
\end{equation}
It is easy to see that \eqref{odepsi} admits a unique solution $X_{\psi}$. We have to prove that $X_\psi$ is measure preserving. A computation like \eqref{mpm} shows that $X_{\psi}(t,\mathbf{E})\# \mathscr{L}^3=\mathscr{L}^3(\mathbf{E})$ for all sets $\mathbf{E}$ of the form
\begin{equation}\label{cs}
\mathbf{E}=\{ (r,\theta,z): \theta\in E_1, r\in [0,\sqrt{z}], z\in E_2 \}
\end{equation}
where $E_1\subset[0,2\pi],E_2\subset\R$.
Sets of the form \eqref{cs} are a basis for the Borel $\sigma$-algebra, hence $X_\psi$ preserve the 3D Lebesgue measure on Borel sets. Since $X_\psi$ maps null sets into null sets, it follows that it is a regular Lagrangian flow.
\end{proof}

\section{Some ideas and heuristics on possible extensions}
Consider the maps
$$
\psi_1(\theta)=\begin{cases} \theta \hspace{1cm} &\mathrm{if} \ \theta\in[0,\pi), \\ 3\pi-\theta  \ \ & \mathrm{if} \ \theta\in[\pi,2\pi], \end{cases}
$$
and 
$$
\psi_2(\theta)=\begin{cases} 2\theta \hspace{1cm} &\mathrm{if} \ \theta\in[0,\pi), \\ 2(\theta-\pi)  \ \ & \mathrm{if} \ \theta\in[\pi,2\pi]. \end{cases}
$$
The map $\psi_1$ leaves half a circle fixed and flips the other half, while the map $\psi_2$ rotates twice around $\mathbb{S}^1$. Since the strategy of the proof of Theorem \ref{teo:main2} produces in the limit only solutions of the form $X^\Theta$, we wonder if it is possible to obtain, as limit of a suitable approximation, the flows $X_{\psi_1},X_{\psi_2}$ associated to $\psi_1,\psi_2$ as in the proof of Proposition 3.2. This is a concrete example of the following general question:
\begin{itemize}
\item[(Q2)]\emph{Does the approximation procedure obtained by smoothing the vector field select a subset of the flows of $b$?}
\end{itemize}
The strategy of \cite{CCS} selects the regular Lagrangian flows corresponding to measure preserving map of the form $\psi(\theta)=\theta+\Theta \ \mathrm{mod \ } 2\pi$. These flows are in a sense ``better" than the others for the following reasons:
\begin{itemize}
\item the flows $X^\Theta$ self intersect only in the origin, while this is not true for $X_{\psi_2}$, which is not even a.e. invertible;
\item the Jacobian of $X^\Theta$ does not change sign, while this is the case for $X_{\psi_1}$.
\end{itemize} 
Consider a general smooth approximation $b_\e$ of the vector field $b$; the corresponding Cauchy problem admits a uniquely defined sequence of flows $X^\e$ and one can ask to which $X_\psi$ the sequence $X^\e$ may converge. It is not clear to us if it is possible to construct an approximation of $b$ in such a way that the approximated flow converge to $X_{\psi_1}$ or $X_{\psi_2}$, especially if we want to approximate $b$ only close to the singularity at the origin. We can however provide some heuristics motivating why it is not trivial to exclude the possibility of getting $X_{\psi_1}$ in the limit just by arguing on the base of "topological obstructions". In fact, we can approximate the flow $X_{\psi_1}$ with maps $X^\e$ of the form:
$$
X^\e(t,\mathbf{x})=
\begin{cases}
X(t,\mathbf{x}) \ & \mathrm{for} \ 0\leq t \leq t_1^\e:=\frac{z^2-\e^2}{4}, \\
\frac{t-t_1^\e}{t_2^\e-t^\e_1}I_-\left(\psi_1(I^{-1}_+(\mathbf{x}))\right)+\frac{t^\e_2-t}{t^\e_2-t^\e_1}X(t^\e_1,\mathbf{x})\ & \mathrm{for} \ t^\e_1\leq t\leq t^\e_2:=\frac{z^2}{4} +\frac{\e^2}{4}, \\
X\left(t-t_2^\e,I_-\left(\psi_1(I^{-1}_+(\mathbf{x}))\right)\right) \ & \mathrm{for} \ t^\e_2\leq t <\infty,
\end{cases}
$$
where $\mathbf{x}\in P^+$. Each $X^\e$ is a well-defined map, which is however not a flow a vector field. Therefore, this does not answer our question. However, this example tells us that an answer in the positive to our question could not just rely on topological properties of the approximating flows.

\subsection*{Acknowledgments}
This research has been supported by the ERC Starting Grant 676675 FLIRT.

\end{document}